\documentclass[oneside,reqno]{amsart}
\usepackage[greek,english]{babel}
\usepackage{amsthm}
\usepackage{amsbsy}
\usepackage{amsfonts}
\usepackage{graphicx}
\newtheorem{thm}{Theorem}
\newtheorem{cor}{Corrolarry}
\newtheorem{lem}{Lemma}
\newtheorem{pro}{Proposition}
\newtheorem{rem}{Remark}

\numberwithin{equation}{section} \numberwithin{lem}{section}
\numberwithin{thm}{section} \numberwithin{cor}{section}
\numberwithin{pro}{section} \numberwithin{rem}{section}
\begin{document}
\title[Irregular boundary layer behavior in a Boussinesq fluid]{Analysis of an irregular boundary layer behavior for the steady state flow of a Boussinesq fluid}


\author{Christos Sourdis}
\subjclass{Primary: 34E10, 34E20; Secondary: 34E13.}
 \keywords{boundary layer, singular perturbation, Painlev\'{e} transcendent, Boussinesq
fluid, Lazer-McKenna conjecture}
\address{Department of Mathematics and Applied Mathematics, University of
Crete, GR--714 09 Heraklion, Crete, Greece.}
\email{csourdis@tem.uoc.gr} \maketitle

\dedicatory{\centerline{\emph{Dedicated to the memory of Professor
Paul C. Fife}}}
\begin{abstract}
Using a perturbation approach, we make rigorous the formal
boundary layer asymptotic analysis of Turcotte, Spence and Bau
from the early eighties for the vertical flow of an internally
heated Boussinesq fluid in a vertical channel with viscous
dissipation and pressure work. A key point in our proof is to
establish the non-degeneracy of a special solution of the
Painlev\'{e}-I transcendent. To this end, we relate this problem
to recent studies for the ground states of the focusing nonlinear
Schr\"{o}dinger equation  in an annulus.  We also relate our
result to a particular case of the well known Lazer-McKenna
conjecture from nonlinear analysis.
\end{abstract}

\section{Introduction}
\subsection{The problem}\label{secproblem}

In \cite{turcotte},  Turcotte, Spence and  Bau considered the
vertical flow of an internally heated Boussinesq fluid in a
vertical channel with viscous dissipation and pressure work.
Starting from the basic equations for conservation of mass,
momentum and energy in a compressible fluid, and after making
various appropriate assumptions, they were led to the study of the
following boundary value problem for the steady state flow:
\begin{equation}\label{eqEq}
\left\{
\begin{array}{ll}
  2u''=u^2-A(1-x^2), & x\in (-1,1), \\
    &   \\
  u(-1)=u(1)=0,  &
\end{array}
\right.
\end{equation}
where $A\geq 0$ is a parameter. More specifically, the parameter
$A$ represents the (non-dimensional) heat addition, $u$ is the
velocity, $x$ is the scaled position  and $[-1,1]$ is the
horizontal cross section of the vertical channel (we refer to
\cite{turcotte} for more details).

\subsection{Formal asymptotic analysis as $A\to \infty$}
A formal asymptotic analysis carried out in \cite{turcotte}
predicts the existence of solutions to (\ref{eqEq}) which, as
$A\to \infty$, behave roughly in the following way. They converge
uniformly to $\sqrt{A(1-x^2)}$ over fixed compacts of $(-1,1)$;
they converge (in some sense) to
\begin{equation}\label{eqh1}
(2A)^\frac{2}{5}Y\left((2A)^\frac{1}{5}(x+1)\right)\ \textrm{and}\
 (2A)^\frac{2}{5}Z\left((2A)^\frac{1}{5}(1-x)\right)\
\end{equation}
near $x=-1$ and $x=1$ respectively, where $Y$ and $Z$ should
satisfy the following boundary value problem:
\begin{equation}\label{eqpainleve}
\left\{\begin{array}{ll}
        2y''=y^2-s, & s>0, \\
         &  \\
        y(0)=0, & y-s^\frac{1}{2}\to 0\ \textrm{as}\ s\to \infty.
      \end{array}
 \right.
\end{equation}

To convince the skeptical reader, let us note that, letting
\[
\varepsilon=\sqrt{\frac{2}{A}}\ \textrm{and}\
v=\frac{1}{\sqrt{A}}u,
\]
problem (\ref{eqEq}) is equivalent to the singular perturbation
problem:
\begin{equation}\label{eqEqSing}
\left\{
\begin{array}{ll}
  \varepsilon^2 v''=v^2- (1-x^2), & x\in (-1,1), \\
    &   \\
  v(-1)=v(1)=0,  &
\end{array}
\right.
\end{equation}
to which one can apply standard, but non-rigorous, matching
asymptotic techniques (see for instance \cite{miller}). In these
terms, $\sqrt{1-x^2}$ serves as an outer solution which, however,
has an \emph{irregular boundary layer}, thus creating the need for
the inner solutions in (\ref{eqh1}).

The familiar reader may have already observed that, after a simple
normalization, the differential equation in (\ref{eqpainleve}) is
non other than the Painlev\'{e}-I transcendent (see for example
\cite[Ch. 5]{fokas}). Despite of this fact, the study of the limit
problem (\ref{eqpainleve}) is nontrivial and, in fact, has quite a
history (see Proposition \ref{proYpm} herein for more details).
Combining the results of \cite{hastingsTroy,holmespainleve}, we
know that problem (\ref{eqpainleve}) has exactly two solutions:
$Y_+$ which is strictly increasing; $Y_-$ which has negative slope
at the origin and exactly one local minimum.

Actually, analogous formal considerations leave open the
possibility of existence of solutions converging to
$-\sqrt{A(1-x^2)}$, in some sense, as $A\to \infty$, but this case
lies beyond the scope of the present article.

\subsection{Rigorous known results}\label{subknownYo}
To the best of our knowledge, there are no studies of the problem
(\ref{eqEq}) that link it rigorously to the limit problem
(\ref{eqpainleve}). On the other hand, we were truly surprised
when we realized the striking similarities that the former problem
shares with a class of extensively studied superlinear elliptic
problems of Ambrosetti-Prodi type and the  famous Lazer-McKenna
conjecture that accompanies them. Interestingly enough, however,
both problems date  to the early 1980's. In particular, for the
simplified problem:
\begin{equation}\label{eqlazer}
\left\{\begin{array}{ll}
  -\Delta u=|u|^p-A\varphi_1(x) & \textrm{in}\ \Omega, \\
    &   \\
  u=0 & \textrm{on}\ \partial \Omega, \\
\end{array} \right.
\end{equation}
where $\Omega$ is a smooth, bounded domain of $\mathbb{R}^N$,
$p\in \left(1,\frac{N+2}{N-2} \right)$ if $N\geq 3$, $p\in
(1,\infty)$ if $N=1,2$, and $\varphi_1>0$ is the principal
eigenfunction of $-\Delta$ in $\Omega$ with Dirichlet boundary
conditions, the Lazer-McKenna conjecture asserts roughly that the
number of solutions diverges as $A\to \infty$ (see
\cite{dancer-lazer,danceryanCrtitic} and the references therein
for more details). Remarkably, the same was also conjectured in
\cite{turcotte} for problem (\ref{eqEq}) and was subsequently
verified by Hastings and McLeod in \cite{hastingsTroy} via a
shooting argument. Let us point out that solutions to
(\ref{eqlazer}) should also develop an irregular boundary layer,
as $A\to \infty$, since the gradient of $\varphi_1$ on $\partial
\Omega$ is nonzero (by Hopf's boundary point lemma). In
\cite{dancer-lazer}, Dancer and Yan proved the Lazer-McKenna
conjecture for (\ref{eqlazer}) by constructing solutions with an
arbitrary number of sharp downward spikes, located near the
maximum points of $\varphi_1$ and superimposed on a positive
minimizing solution, provided that $A$ is sufficiently large. They
also studied the asymptotic behavior, as $A\to \infty$, of the
mountain pass solutions to (\ref{eqlazer}) and showed that they
have a small steep peak near the boundary (combined with the
irregular boundary layer). In connection with this, let us note
that the aforementioned increasing solution $Y_+$ of
(\ref{eqpainleve}) is a minimizer of the natural associated
energy, while the other solution $Y_-$, which has a negative peak,
is a mountain pass. Even though the irregular boundary layer of
the problem (\ref{eqlazer}) was treated mostly as a tangential
issue in \cite{dancer-lazer}, the author still had to study the
elliptic analog of (\ref{eqpainleve}) (with exponent $p$ in the
nonlinearity).

It follows readily from the analysis in \cite{dancer-lazer}, which
was variational in nature, that problem (\ref{eqEq}) has two even
solutions $u_\pm$ such that
\begin{equation}\label{eqDYout}
u_{\pm}(x)=\sqrt{A(1-x^2)}-(1-x^2)^{-2}+A^{-\frac{1}{2}}\mathcal{O}\left(1
\right), \end{equation} uniformly over fixed compacts of $(-1,1)$,
as $A\to \infty$ (here, and throughout this paper, Landau's symbol
$\mathcal{O}(1)$ denotes a quantity which is bounded independently
of large $A$);
\begin{equation}\label{eqDYin}
(2A)^{-\frac{2}{5}}u_{\pm}\left(-1+(2A)^{-\frac{1}{5}}s \right)\to
Y_\pm (s)\ \textrm{in}\ C_{loc}[0,\infty),\ \textrm{as}\ A\to
\infty.
\end{equation}
Actually, only even solutions were considered in \cite{turcotte}.
As may be expected, the minus case in the above result is
considerably harder to establish and, for this purpose, the
authors had to adapt some ideas from \cite{delPinoIndi}.  Let us
emphasize that $u_-$ has to be constructed as a mountain pass in
the class of even functions. Indeed, the associated linearization
on $u_-$, for large $A$, has similarities with a semiclassical
Schr\"{o}dinger operator with a double-well potential; thus, it is
expected to have at least two unstable eigenvalues in the
non-symmetric class (by the instability of $u_-$ in the symmetric
class and tunnelling phenomena, see \cite{hislop} and Remark
\ref{remTunel} below), contradicting \cite{hofer} should $u_-$ was
a mountain pass solution in the non-symmetric class. The mountain
pass solutions, in the general class, are expected to have the
(re-scaled) profile of $Y_-$ at one boundary point and that of
$Y_+$ at the other, for large $A$. However, this does not seem to
follow directly from the analysis in \cite{dancer-lazer}. Let us
also point out that the variational approach, used for showing the
above, does not require any knowledge of the non-degeneracy of the
solutions $Y_\pm$ of (\ref{eqpainleve}), that is the absence of
bounded elements in the kernel of the associated linearizations.
On the other hand, it is essentially the non-degeneracy of the
corresponding $Y_+$ that allowed the authors of
\cite{dancer-lazer} to add sharp downward spikes on top of the
corresponding minimal solution $u_+$ by means of a finite
dimensional variational reduction procedure.

\subsection{The main result}
In this article, using a perturbation argument, we will give
optimal estimates for the convergence in (\ref{eqDYin}) and also
provide the missing estimates in the intermediate zones that are
not covered by (\ref{eqDYout}) and (\ref{eqDYin}). In the process,
we will prove the non-degeneracy of the ``blow-up'' profile $Y_-$,
and at the same time provide a new proof of the fact that
(\ref{eqpainleve}) has only $Y_+$ and $Y_-$ as solutions, which
was originally shown in \cite{hastingsTroy}.

The following is our main result.
\begin{thm}\label{thmex}
Let $Y$ and $Z$ be either one of the two solutions $Y_+$ and $Y_-$
of (\ref{eqpainleve}). There exists a solution $u=u_{YZ}$ of
(\ref{eqEq}) such that
\[\left\{\begin{array}{ll}
    u=(2A)^\frac{2}{5}Y\left((2A)^\frac{1}{5}(x+1)\right)+\mathcal{O}(A^\frac{2}{5})(x+1), & 0\leq x+1\leq (2A)^{-\frac{1}{5}}D, \\
     &  \\
    u=(2A)^\frac{2}{5}Y\left((2A)^\frac{1}{5}(x+1)\right)+\mathcal{O}(A^\frac{1}{2})(x+1)^\frac{3}{2}, & (2A)^{-\frac{1}{5}}D\leq x+1\leq \delta,
  \end{array}\right.
\]
\[\left\{\begin{array}{ll}
    u=(2A)^\frac{2}{5}Z\left((2A)^\frac{1}{5}(1-x)\right)+\mathcal{O}(A^\frac{2}{5})(1-x), & 0\leq 1-x\leq (2A)^{-\frac{1}{5}}D, \\
     &  \\
    u=(2A)^\frac{2}{5}Z\left((2A)^\frac{1}{5}(1-x)\right)+\mathcal{O}(A^\frac{1}{2})(1-x)^\frac{3}{2}, & (2A)^{-\frac{1}{5}}D\leq 1-x\leq \delta,
  \end{array}\right.
\]
and
\[
u=\sqrt{A(1-x^2)}+\mathcal{O}(1)(1-x^2)^{-2},\ \ x\in
\left[-1+(2A)^{-\frac{1}{5}}D,1-(2A)^{-\frac{1}{5}}D \right],
\]
for some constants $0<\delta\ll D$, uniformly as $A\to \infty$
(for the above notation, see Subsection \ref{secnotation} below).

Moreover, the following a-priori estimate holds for the associated
linearized operator: There exist constants $A_1, C>0$ such that if
$\varphi\in C^2[-1,1]$ and $f\in C[-1,1]$ satisfy
\[
\left\{
\begin{array}{cc}
  -\varphi''+u_{YZ}\varphi=f, & x\in (-1,1), \\
    &   \\
  \varphi(-1)=0=\varphi(1), &   \\
\end{array}
\right.
\]
for $A\in (A_1,\infty)$, then
\[
\|\varphi\|_{L^\infty(-1,1)}\leq C
A^{-\frac{2}{5}}\|f\|_{L^\infty(-1,1)}.
\]
\end{thm}

Although we do not show it, the above estimates are optimal as can
be easily verified by simple scaling arguments.

In the case of even solutions, it turns out that $u_{Y_+Y_+}$ is
asymptotically stable while $u_{Y_-Y_-}$ is unstable with Morse
index equal to two (see Remarks  \ref{remTunel},
\ref{remlinearnearby} and \ref{remeven} below). In the
nonsymmetric case, we can tell that $u_{Y_-Y_+}$ and $u_{Y_+Y_-}$
have Morse index one (see Remark \ref{remunbalanced} below).

We expect that the above a-priori estimate for the linearized
operator can allow to extend the usual variational reduction
procedure, similarly to \cite{dancer-lazer}, in order to construct
new solutions to (\ref{eqEq}), having an arbitrary number of
downward spikes near the origin (each of scale $A^{-\frac{1}{2}}$
and at an $\mathcal{O}\left((\ln A)A^{-\frac{1}{2}} \right)$
distance from the others) that are superimposed on the profile of
$u_{YZ}$, for large $A>0$.

The proof of Theorem \ref{thmex} carries over directly to the case
of the singular perturbation problem
\[
\left\{\begin{array}{cc}
  \varepsilon^2 u''=F(u,x), & x\in (a,b), \\
    &   \\
  u(a)=0=u(b), &   \\
\end{array} \right.
\]
provided that the following assumptions are met: $F\in
C^3\left(\mathbb{R}\times [a,b] \right)$ and there exists a
$u_0\in C^2(a,b)\cap C[a,b]$ such that $u_0(a)=0=u_0(b)$,
\[
F\left(u_0(x),x \right)=0,\ x\in [a,b], \ \ F_u\left(u_0(x),x
\right)>0,\ x\in (a,b),
\]
\[
F_u=0,\ F_x<0,\ F_{uu}>0,\ F_{ux}=0\ \textrm{at}\ (0,a),
\]
\[
F_u=0,\ F_x>0,\ F_{uu}>0,\ F_{ux}=0\ \textrm{at}\ (0,b).
\]

\subsection{Method of proof}
Our strategy is to apply a perturbation argument that has been
used in many papers in the last years. This type of argument
consists of three main steps: Firstly, one constructs a
sufficiently good approximate solution to the problem, then
studies the invertibility properties of the associated
linearization about this approximation, and finally captures a
true solution that is close, in some sense, to the approximate one
by some type of fixed point argument. This approach, however,
relies heavily on the good understanding of the corresponding
limit problems, something which is not the case here since the
non-degeneracy of $Y_-$ does not seem to be known. In addition,
the solutions that we expect to find are not localized in the
conventional sense, as they should develop irregular boundary
layers. In this regard, let us point out that an extra difficulty
is that the convergence of $Y_\pm$ to the square root profile is
algebraically slow (see (\ref{eqY-sqrt}) below).

We are able to prove the non-degeneracy of $Y_-$ by reducing
(\ref{eqpainleve}) to the ground state problem for a nonlinear
Schr\"{o}dinger equation in the half-line with zero boundary
conditions (see (\ref{eqeta}) below), and take advantage of the
many studies that have been conducted on uniqueness and
non-degeneracy issues for the latter problem (see
\cite{byeonOshita,felmerUniqueness,tanaka,yotsutani}). In fact,
this also allows us to give a new proof of the uniqueness of
$Y_{\pm}$, as was originally conjectured in \cite{holmespainleve}
and proven by completely different techniques in
\cite{hastingsTroy}. Armed with the knowledge of the
non-degeneracy of the blow-up profiles $Y_\pm$, we can deal with
the difficulties related to the irregular boundary layer behavior
by adapting the perturbative approach that was developed in the
recent papers \cite{karaliGS,sourdis-fife}, where the
corresponding blow-up problem featured the Painlev\'{e}-II
transcendent.

Let us point out that, in contrast to the problems in the
aforementioned references, the instability of $Y_-$ (recall our
discussion in Subsection \ref{subknownYo}) suggests that the
 solutions of (\ref{eqEq}) with this blow-up profile in
one of the boundaries should also be unstable. Therefore, the well
known method of upper and lower solutions (barriers), see for
example \cite{sattinger}, should not be applicable to capture such
 solutions of (\ref{eqEq}).

\subsection{Relations with geometric singular perturbation theory}
The ordinary differential equation in (\ref{eqEqSing}) can be
written as a three-dimensional, slow-fast system (see \cite{jo}),
having a one-dimensional slow manifold which undergoes saddle-node
bifurcations at the values $\pm 1$ of the slow variable $x$. In
light of the non-degeneracy of $Y_-$ that we will prove, it seems
plausible that our main result can also be proven by the blow-up
approach to geometric singular perturbation theory (see
\cite{schecter-sourdis} for a related problem with one turning
point that involves the Painlev\'{e}-II transcendent).

\subsection{Extensions} We expect that an analogous result to
Theorem \ref{thmex} holds for positive solutions to
(\ref{eqlazer}) (at least for $p\geq 2$), having  the profile of
the corresponding to $Y_+$ one-dimensional stable solution  of the
blow-up problem
\[
\left\{\begin{array}{ll}
  y_{ss}+\Delta_{\mathbb{R}^{N-1}}=|y|^p-s, & (s,\theta)\in (0,\infty)\times \mathbb{R}^{N-1}, \\
    &   \\
  y(0,\theta)=0,\ \ y(s,\theta)-s^\frac{1}{p}\to 0  & \textrm{as}\ \ s\to \infty,\ \ \textrm{uniformly\ in} \ \theta\in \mathbb{R}^{N-1}.  \\
\end{array} \right.
\]
orthogonal to the boundary.

In view of the preceding discussion and the results of the current
article, the only other solution of the above blow-up problem for
which we have some non-degeneracy information is $Y_-(s)$ for the
case $p=2$. However, note that this solution has infinite Morse
index as a solution of the above problem, and thus a corresponding
perturbation result should involve resonance phenomena (see
\cite{delpinocpam,malchiodiCPAM}, and especially
\cite{karalisourdisresonance} where solutions exhibiting similar
irregular layered behavior were studied).

\subsection{Outline of the paper}
In Section \ref{secuap}, we will construct sufficiently good
approximate solutions to (\ref{eqEq}). This is the main section of
the paper and it is where we will prove the non-degeneracy of
$Y_-$ (the full details will be postponed to  Appendix
\ref{appenFelmer}). In Section \ref{seclinear}, we will study the
invertibility properties of the linearization of (\ref{eqEq}) on
the constructed approximate solutions, relying heavily on the
non-degeneracy of $Y_\pm$. In Section \ref{secexist}, we will use
the obtained linear estimates to perturb the approximate solutions
to genuine ones, and also obtain related estimates for their
difference by various comparison arguments. Finally, in Section
\ref{secmainResult}, we will combine everything together to prove
Theorem \ref{thmex}. We will close the paper with an appendix,
providing the full details of the proof of the non-degeneracy of
$Y_-$.

\subsection{Notation}\label{secnotation} In the sequel, we will
often suppress the obvious dependence on $A$ of various functions
and quantities. Furthermore, by $c/C$ we will denote small/large
generic constants, independent of $A$, whose value will change
from line to line. The value of $A$ will constantly increase so
that all previous relations hold. The Landau symbol
$\mathcal{O}(1), \ A\to \infty$, will denote quantities that
remain uniformly bounded  as $A \to \infty$, whereas $o(1)$ will
denote quantities that approach zero as $A\to \infty$.

\section{Construction of an approximate solution
$u_{ap}$}\label{secuap}

In this section, we will construct  sufficiently good approximate
solutions to the problem (\ref{eqEq}) with the same type of
behavior as the solutions that we are looking for.
\subsection{The inner (boundary layer) solution $u_{in}$}

In this subsection, motivated from the aforementioned formal
analysis in \cite{turcotte}, we will use the solutions $Y_\pm$ of
the blow-up problem (\ref{eqpainleve}) to construct approximate
solutions to (\ref{eqEq}) which, however, are effective only near
the boundary of the interval.

The properties of the blow-up profiles $Y_\pm$ that we will need
for the purposes of this paper are contained in the following
proposition.

\begin{pro}\label{proYpm}
The boundary value problem (\ref{eqpainleve}) has exactly two
solutions $Y_+$ and $Y_-$. We have that $(Y_+)'>0$ in
$[0,\infty)$, while $(Y_-)'(0)<0$ and $Y_-$ has a unique minimum
at some point in $(0,\infty)$. Moreover, the solutions $Y_\pm$ are
non-degenerate, in the sense that there do not exist nontrivial
bounded smooth solutions of
\begin{equation}\label{eqnondegYpm}
\psi''-Y_\pm \psi=0\ \ \textrm{in}\ [0,\infty),\ \ \varphi(0)=0.
\end{equation}
\end{pro}
\begin{proof}
Existence of two solutions, satisfying the monotonicity properties
described in the assertion of the proposition, has been
established by Holmes and Spence \cite{holmespainleve} by a
shooting argument (and in \cite{dancer-lazer}, via the method of
upper/ lower solutions and variational arguments, perhaps unaware
of \cite{holmespainleve}). The authors of \cite{holmespainleve}
also conjectured that these solutions were indeed the only ones.
Their conjecture was settled, to the affirmative, by Hastings and
Troy \cite{hastingsTroy}. However, their proof was, as we discover
(\emph{almost 25 years later!}), much more complicated than
necessary, and relied on some four decimal point numerical
calculations. A truly simple proof of the uniqueness result of
\cite{hastingsTroy}, which in the process implies the desired
non-degeneracy of solutions  can be given, based on a previous
remark of ours from \cite{karaliGS} (see Remark 35 therein), as
follows. It is easy to see that $Y_+$ is non-degenerate and the
unique increasing solution of (\ref{eqpainleve}) (see for example
\cite{dancer-lazer}). Let $\tilde{Y}$ be any other solution of
(\ref{eqpainleve}), and let $u=Y_+-\tilde{Y}$. By an easy
calculation, and the maximum principle, we find that $u$ has to be
a  solution of
\begin{equation}\label{eqeta}
2u''-2Y_+(s)u+u^2=0,\ s>0,\ u(s)>0,\ s>0,\ \ u(0)=0,\ u(s)\to 0\
\textrm{as}\ s\to \infty.
\end{equation}

It has been shown recently in \cite{felmerUniqueness} that the
general problem
\begin{equation}\label{eqfelmerGener}\left\{\begin{array}{l}
    u''+\frac{\nu}{s}u'-V(s)u+u^p=0,\ \ s>a, \\
     \\
   u(s)>0,\ s>a,\ \ u(a)=0,\ u(s)\to 0\ \textrm{as}\ s\to \infty,
  \end{array}\right.
\end{equation}
where $a>0$, $\nu\geq 0$, $p\in (1,\infty)$ and $V\in C^1\left(
[a,\infty)\right)$, has at most one solution provided that the
auxiliary function
\[
U(s)=V'(s)s^3+\beta V(s)s^2+(\beta-2)L,
\]
with
\[
\alpha=\frac{2\nu}{p+3},\ \ \beta=(p-1)\alpha, \ \textrm{and}\ \
L=\alpha(\nu-1-\alpha),
\]
satisfies
\[
\liminf_{s\to \infty}U(s)>0,
\]
and one of the following conditions:
\begin{description}
  \item[(i)] $U$ is positive in $(a,\infty)$,
  \item[(ii)] $U(a)<0$ and $U$ changes sign only once in
  $(a,\infty)$.
\end{description}
Moreover, if we further assume that $\nu>0$, it has been shown in
the same reference that the unique solution of (\ref{eqfelmerGener})
is non-degenerate (if such solution exists).

We would like to adapt the proof of the aforementioned result in
order to establish that the solution $Y_+-Y_-$ of (\ref{eqeta}) is
unique and non-degenerate. Comparing with (\ref{eqfelmerGener}), in
the problem at hand (\ref{eqeta})  we have $a=0$, $p=2$, $\nu=0$,
$\alpha=0$, $V(s)=Y_+(s)$, $s>0$, and $U(s)=s^3Y_+'(s)$, $s>0$. We
note that, by scaling, the result of \cite{felmerUniqueness}
continues to hold when a positive constant multiplies the power
nonlinearity. Observe that the corresponding case (i) above holds.
However, our potential $V=Y_+$ loses its positivity at $s=0$ and
also becomes unbounded as $s\to \infty$. On top of that, in our case
$\nu=0$ and not positive as required in \cite{felmerUniqueness} for
showing the non-degeneracy of the solution. Nevertheless, as we will
see, the proof of \cite{felmerUniqueness} can be easily  adapted to
establish uniqueness for the problem (\ref{eqeta}), and with some
care the same can be done for showing that the solution $Y_+-Y_-$ is
non-degenerate. This implies at once that (\ref{eqpainleve}) has
exactly the two solutions $Y_+$ and $Y_-$, while the desired
non-degeneracy property of the solution $Y_-$ follows readily. In
Proposition \ref{profelmer} of Appendix \ref{appenFelmer} we will
indicate how the arguments of \cite{felmerUniqueness} can be adapted
to provide uniqueness and non-degeneracy for (\ref{eqeta}).

The proof of the proposition is complete.

\end{proof}

\begin{rem}
In \cite{karalisourdisradial,karalisourdisresonance} we had
previously applied the same idea, used in the proof of Proposition
\ref{proYpm} for the study of $Y_-$, to the problem
\[
y''=y^2-s^2,\ \ s\in \mathbb{R};\ \ y(s)-|s|\to 0\ \textrm{as}\
|s|\to \infty,
\]
and showed that it has exactly two solutions, one which  is stable
and another one which is unstable. Interestingly enough, during
the preparation of the current article, we came across the paper
\cite{holmes} where the same result was previously obtained by
different techniques (which are similar to those that were
subsequently used in \cite{holmespainleve}). Actually, the
non-degeneracy of the unstable solution to the above problem,
which we also proved in \cite{karalisourdisradial} and enabled us
to carry out the corresponding perturbation analysis, does not
seem to be contained in \cite{holmes}.
\end{rem}
\begin{rem}
In \cite{holmespainleve} it was also shown that there are
solutions to (\ref{eqpainleve}) that approach $-\sqrt{s}$, instead
of $\sqrt{s}$, as $s\to \infty$. For a thorough analysis of such
solutions and more up to date references, we refer the interested
reader to \cite{hastingsBOOK}.
\end{rem}

It is easy to show that
\begin{equation}\label{eqY-sqrt}
Y(s)-s^\frac{1}{2}=\mathcal{O}(s^{-2})\ \textrm{as}\ s\to\infty,
\end{equation}
and
\begin{equation}\label{eqY'}
Y'(s)=\frac{1}{2}s^{-\frac{1}{2}}+\mathcal{O}(s^{-3}),\ \
Y''(s)=\mathcal{O}(s^{-\frac{3}{2}})\ \ \textrm{as}\ \ s\to
\infty,
\end{equation}
(see also \cite[App. A]{sourdis-fife}).

Let $Y,\ Z$ denote either one of $Y_\pm$, we define the inner
solution of (\ref{eqEq}) near $x=-1$ as
\begin{equation}\label{equinner-}
u_{in}(x)=(2A)^\frac{2}{5}Y\left((2A)^\frac{1}{5}(x+1)\right)\
\textrm{for}\ 0\leq s\equiv (2A)^\frac{1}{5}(x+1)\leq
\delta(2A)^\frac{1}{5},
\end{equation}
 where $\delta>0$ is a small constant
independent of $A$. Similarly, close to $x=1$, we define
\begin{equation}\label{equinner+}
u_{in}(x)=(2A)^\frac{2}{5}Z\left((2A)^\frac{1}{5}(1-x)\right)\
\textrm{for} \ 0\leq t\equiv (2A)^\frac{1}{5}(1-x)\leq
\delta(2A)^\frac{1}{5}.
\end{equation}

The effectiveness of $u_{in}$ as an approximate solution can be
mainly measured from the estimate in the following proposition.

\begin{pro}\label{prouin}
We have
\[
2u_{in}''-u_{in}^2+A(1-x^2)=\mathcal{O}(A)(1-x^2)^2\ \textrm{as}\
A\to \infty,
\]
uniformly on $[-1,-1+\delta]\cup [1-\delta,1]$.
\end{pro}
\begin{proof}
We will sketch the proof in the case where $x\in [-1,-1+\delta]$,
the other case can be treated identically. The desired estimate
follows readily by linearizing $1-x^2$ at $x=-1$, which reads as
\[
1-x^2=2(x+1)-(x+1)^2,
\]
and using (\ref{eqpainleve}).

 The proof of the proposition is complete.\end{proof}

\subsection{The modified outer solution $\tilde{u}_{out}$} Instead of
using the outer solution
\[
u_{out}=\sqrt{A(1-x^2)},
\]
 we will use a more sophisticated approximation
\begin{equation}\label{equoutmod}
\tilde{u}_{out}=\left\{A(1-x^2)-(2A)^\frac{4}{5}\left[s-Y^2(s)\right]n_\delta(1+x)-(2A)^\frac{4}{5}\left[t-Z^2(t)\right]n_\delta(1-x)
\right\}^\frac{1}{2},
\end{equation}
$x\in [-1,1]$, where $s,t$ as in (\ref{equinner-}),
(\ref{equinner+}) respectively (but now defined on $[0,\infty)$),
and $n_\delta$ is a smooth cutoff function such that
\begin{equation}\label{eqcutoff}
n_d(r)=\left\{\begin{array}{ll}
                     1 & \textrm{if}\ |r|\leq d, \\
                       &   \\
                     0 & \textrm{if}\ |r|\geq 2d.
                   \end{array}
 \right.
\end{equation}

Our motivation for the definition of $\tilde{u}_{out}$ comes from
\cite{karalisourdisresonance,karaliGS}. However, let us note that
formulas of related nature can be found (at the formal level) in
some books of asymptotic analysis (see \cite[Ch. 8]{miller}).

The main result concerning $\tilde{u}_{out}$ is the following.

\begin{pro}\label{prouout}
We have
\begin{equation}\label{equoutremain}
2\tilde{u}_{out}''-\tilde{u}_{out}^2+A(1-x^2)=\mathcal{O}(A^\frac{1}{2})(1-x^2)^{-\frac{1}{2}},
\end{equation}
uniformly on $[-1+\delta^{-1}(2A)^{-\frac{1}{5}},
1-\delta^{-1}(2A)^{-\frac{1}{5}}]$, as $A\to \infty$. Moreover, we
find that
\begin{equation}\label{equout-uin}\begin{array}{c}
 \tilde{u}_{out}-u_{in}=\mathcal{O}(A^\frac{1}{2})(1-x^2)^\frac{3}{2},\
\left(\tilde{u}_{out}-u_{in}\right)'=\mathcal{O}(A^\frac{1}{2})(1-x^2)^\frac{1}{2}, \\
    \\
  \left(\tilde{u}_{out}-u_{in}\right)''=\mathcal{O}(A^\frac{1}{2})(1-x^2)^{-\frac{1}{2}},
 \\
\end{array}
\end{equation}
uniformly on $[-1+\delta^{-1}(2A)^{-\frac{1}{5}},-1+\delta]\cup
[1-\delta,1-\delta^{-1}(2A)^{-\frac{1}{5}}]$, as $A\to \infty$.
\end{pro}
\begin{proof}
If  $x\in[-1+\delta^{-1}(2A)^{-\frac{1}{5}},-1+\delta]$, recalling
(\ref{equoutmod}) and (\ref{eqcutoff}), we obtain that
\begin{equation}\label{eqeasyway}
\tilde{u}_{out}^2-A(1-x^2)=(2A)^\frac{4}{5}\left[Y^2(s)-s
\right]=u_{in}''.
\end{equation}
 In the same
interval, we can write
\[
\tilde{u}_{out}(x)= (2A)^\frac{2}{5}
\left[-\frac{1}{2}(2A)^\frac{1}{5}(x+1)^2+Y^2\left((2A)^\frac{1}{5}(x+1)\right)\right]^\frac{1}{2}.
\]
Hence, by (\ref{eqY-sqrt}) and (\ref{equinner-}), we get
\[
\tilde{u}_{out}=(2A)^\frac{2}{5}Y(s)\left[1-\frac{1}{2}(2A)^{-\frac{1}{5}}s^2Y^{-2}
\right]^\frac{1}{2}=(2A)^\frac{2}{5}Y(s)\left[1+\mathcal{O}(A^{-\frac{1}{5}})s\right]=u_{in}+
\mathcal{O}(A^\frac{1}{5})s^\frac{3}{2},
\]
uniformly for $s\in[\delta^{-1}, \delta (2A)^\frac{1}{5}]$, as $A\to
\infty$.  Moreover, direct differentiation yields that
\[
\begin{array}{lll}
  \tilde{u}_{out}' & = & \frac{(2A)^\frac{2}{5}}{2}
\left[-(2A)^\frac{1}{5}(x+1)+2(2A)^\frac{1}{5}YY'\right]
\left[-\frac{1}{2}(2A)^\frac{1}{5}(x+1)^2+Y^2\right]^{-\frac{1}{2}} \\
    &   &   \\
    &  = &\frac{(2A)^\frac{2}{5}}{2}
\left[-s+2(2A)^\frac{1}{5}YY'(s)\right]
\left[-\frac{1}{2}(2A)^{-\frac{1}{5}}s^2+Y^2(s)\right]^{-\frac{1}{2}},
\end{array}
\]
and
\[
\begin{array}{lll}
  \tilde{u}_{out}'' & = & \frac{(2A)^\frac{2}{5}}{2}
\left[-(2A)^\frac{1}{5}+2(2A)^\frac{2}{5}(Y')^2(s)+2(2A)^\frac{2}{5}Y''Y(s)\right]\\
& & \times
\left[-\frac{1}{2}(2A)^{-\frac{1}{5}}s^2+Y^2(s)\right]^{-\frac{1}{2}}
 \\
    &   &   \\
    &   & -\frac{(2A)^\frac{2}{5}}{4}\left[-s+2(2A)^\frac{1}{5}YY'(s)
\right]^2\left[-\frac{1}{2}(2A)^{-\frac{1}{5}}s^2+Y^2(s)\right]^{-\frac{3}{2}}.
\end{array}
\]
By (\ref{eqY-sqrt}), (\ref{eqY'}), and (\ref{equinner-}), we have
that
\[\begin{array}{lll}
    \tilde{u}_{out}' & = & \left[u_{in}'-\frac{(2A)^\frac{2}{5}}{2}sY^{-1}\right]\left[1-\frac{1}{2}(2A)^{-\frac{1}{5}}s^2Y^{-2}
\right]^{-\frac{1}{2}} \\
      &    &   \\
      &  = & \left[(2A)^\frac{3}{5}Y'(s)-\frac{(2A)^\frac{2}{5}}{2}\mathcal{O}(s^\frac{1}{2})\right]\left[1-\frac{1}{2}(2A)^{-\frac{1}{5}}\mathcal{O}(s)
\right]
 \\
      &   &   \\
      &   =&u_{in}'+ \mathcal{O}(A^\frac{2}{5})s^\frac{1}{2}
  \end{array}
\]
uniformly for $s\in[\delta^{-1}, \delta (2A)^\frac{1}{5}]$, as $A\to
\infty$.  In the same fashion, we  can show that
\[
\tilde{u}_{out}''=u_{in}''+\mathcal{O}(A^\frac{3}{5})s^{-\frac{1}{2}},
\]
uniformly for $s\in[\delta^{-1}, \delta (2A)^\frac{1}{5}]$, as $A\to
\infty$. We point out that
$\tilde{u}_{out}''=\left(\tilde{u}_{out}''\right)_1-\left(\tilde{u}_{out}''\right)_2$,
with the obvious notation, and
\[\begin{array}{l}
 \left(\tilde{u}_{out}''\right)_1=(2A)^\frac{4}{5}Y''+(2A)^\frac{4}{5}(Y')^2Y^{-1}+\mathcal{O}(A^\frac{3}{5})s^{-\frac{1}{2}}, \\
    \\
  \left(\tilde{u}_{out}''\right)_2=(2A)^\frac{4}{5}(Y')^2Y^{-1}+\mathcal{O}(A^\frac{3}{5})s^{-\frac{1}{2}}, \\
\end{array}
\]
uniformly as $A\to \infty$.

In $[-1+\delta, -1+2\delta]$, it follows readily from
(\ref{eqY-sqrt}), (\ref{eqY'}), and (\ref{equoutmod}) that
\[
\tilde{u}_{out}^2-A(1-x^2)=\mathcal{O}(A^\frac{1}{2}),\
\tilde{u}_{out}'=\mathcal{O}(A^\frac{1}{2}),\
\tilde{u}_{out}''=\mathcal{O}(A^\frac{1}{2}),
\]
uniformly as $A\to \infty$ (An easy way to see these is to note that
we have $\tilde{u}_{out}\geq c A^\frac{1}{2}$ and then differentiate
twice (\ref{eqeasyway}) with righthand side multiplied by the
cutoff). Similar estimates hold for the remaining regions of
$[-1,1]$.

The desired assertions of the proposition follow readily from the
above relations.
\end{proof}

The following estimates will be useful in the sequel.

\begin{lem}\label{lemuout-sqrt}
We have
\begin{equation}\label{equout-sqrt}
\tilde{u}_{out}=\sqrt{A(1-x^2)}+\mathcal{O}\left((1-x^2)^{-2}
\right)
\end{equation}
uniformly on
$[-1+\delta^{-1}(2A)^{-\frac{1}{5}},-1+2\delta]\cup[1-2\delta,1-\delta^{-1}(2A)^{-\frac{1}{5}}]$,
and
\begin{equation}\label{equout=sqrt}
\tilde{u}_{out}=\sqrt{A(1-x^2)},\ \ x\in [-1+2\delta,1-2\delta].
\end{equation}
\end{lem}
\begin{proof}
If $x\in [-1+\delta^{-1}(2A)^{-\frac{1}{5}},-1+2\delta]$, which
implies that $s=(2A)^\frac{1}{5}(x+1)\geq \delta^{-1}$, from
(\ref{equoutmod}), via (\ref{eqY-sqrt}), we obtain that
\[
\tilde{u}_{out}=A^\frac{1}{2}(1-x^2)^\frac{1}{2}\left[1+\mathcal{O}(s^{-\frac{5}{2}})
\right]^\frac{1}{2}=A^\frac{1}{2}(1-x^2)^\frac{1}{2}\left[1+\mathcal{O}(s^{-\frac{5}{2}})
\right],
\]
and (\ref{equout-sqrt}) follows readily. Analogously we treat the
case where $x\in [1-2\delta,1-\delta^{-1}(2A)^{-\frac{1}{5}}]$.
Relation (\ref{equout=sqrt}) follows immediately from the
definitions (\ref{equoutmod}) and (\ref{eqcutoff}).

The proof of the lemma is complete.
\end{proof}

\subsection{Gluing the inner and outer approximations in order to create  the global approximation $u_{ap}$}
We define our global approximate solution to (\ref{eqEq}) to be the
smooth function
\begin{equation}\label{equap}
u_{ap}=\left\{
\begin{array}{ll}
  u_{in}, & x\in [-1,-1+\delta^{-1}(2A)^{-\frac{1}{5}}]\cup[1-\delta^{-1}(2A)^{-\frac{1}{5}},1],      \\
   &  \\
   \tilde{u}_{out}+(\chi_-+\chi_+)(u_{in}-\tilde{u}_{out}), &x\in
   [-1+\delta^{-1}(2A)^{-\frac{1}{5}},1-\delta^{-1}(2A)^{-\frac{1}{5}}],
\end{array}
\right.
\end{equation}
with $u_{in},\ \tilde{u}_{out}$ as in
(\ref{equinner-})--(\ref{equinner+}), (\ref{equoutmod})
respectively, and
\begin{equation}\label{eqchi}
\chi_\mp(x)=n_{\delta^{-1}}\left((2A)^\frac{1}{5}(1\pm x) \right),
\end{equation}
where $n_{\delta^{-1}}$ defined through (\ref{eqcutoff}).

The main result concerning $u_{ap}$ is the following.

\begin{pro}\label{prouap}
Letting \[ \mathcal{E}\equiv2u_{ap}''-u_{ap}^2+A(1-x^2),
\] we have
\begin{equation}\label{equapremainder}
\mathcal{E}=\left\{\begin{array}{ll}
                                     \mathcal{O}(A)(1-x^2)^2, & x\in [-1,-1+\delta^{-1}(2A)^{-\frac{1}{5}}]\cup[1-\delta^{-1}(2A)^{-\frac{1}{5}},1],   \\
                                      &    \\
                                     \mathcal{O}(A^\frac{1}{2})(1-x^2)^{-\frac{1}{2}},
                                     & x\in
   [-1+\delta^{-1}(2A)^{-\frac{1}{5}},1-\delta^{-1}(2A)^{-\frac{1}{5}}],
                                   \end{array}
\right.
\end{equation}
uniformly as $A\to \infty$.
\end{pro}
\begin{proof}
Outside of the interpolation region
$[-1+\delta^{-1}(2A)^{-\frac{1}{5}},-1+2\delta^{-1}(2A)^{-\frac{1}{5}}]\cup
[1-2\delta^{-1}(2A)^{-\frac{1}{5}},1-\delta^{-1}(2A)^{-\frac{1}{5}}]$,
relation (\ref{equapremainder}) follows at once from the assertions
of Propositions \ref{prouin} and \ref{prouout}. In
$[-1+\delta^{-1}(2A)^{-\frac{1}{5}},-1+2\delta^{-1}(2A)^{-\frac{1}{5}}]$,
by the estimates of Proposition \ref{prouin} and \ref{prouout}, we
have that
\begin{equation}\label{equapremcalcut}
\begin{array}{lll}
  2u_{ap}''-u_{ap}^2+A(1-x^2) & = & 2\tilde{u}_{out}''-\tilde{u}_{out}^2+A(1-x^2) \\
    &   & +2\chi_-''(u_{in}-\tilde{u}_{out})+4\chi_-'(u_{in}-\tilde{u}_{out})'+2\chi_-(u_{in}-\tilde{u}_{out})'' \\
    &   & -2\tilde{u}_{out}\chi_-(u_{in}-\tilde{u}_{out})-\chi_-^2(u_{in}-\tilde{u}_{out})^2 \\
   &  &  \\
    & =  &\mathcal{O}(A^\frac{1}{2})(1+x)^{-\frac{1}{2}}     \\
    &   & +\mathcal{O}\left(A^\frac{1}{2}(1+x)^\frac{3}{2}A^\frac{2}{5}+A^\frac{1}{5}A^\frac{1}{2}(1+x)^\frac{1}{2}+A^\frac{1}{2}(1+x)^{-\frac{1}{2}} \right)  \\
    &   &  +\mathcal{O}\left(A^\frac{2}{5}A^\frac{1}{2}(1+x)^\frac{3}{2}+A(1+x)^3 \right) \\
    &   &\\
    &=&\mathcal{O}(A^\frac{1}{2})(1+x)^{-\frac{1}{2}},
\end{array}
\end{equation}
uniformly as $A\to \infty$. Analogous estimates hold true in the
interpolation region
$[1-2\delta^{-1}(2A)^{-\frac{1}{5}},1-\delta^{-1}(2A)^{-\frac{1}{5}}]$.

The proof of the proposition is complete.
\end{proof}

We have the following two easy corollaries.

\begin{cor}\label{coruapremainder}
We have
\begin{equation}\label{eqremainderuapGlobal}
\|\mathcal{E}\|_{L^\infty(-1,1)}=\|2u_{ap}''-u_{ap}^2+A(1-x^2)\|_{L^\infty(-1,1)}\leq
CA^\frac{3}{5}.
\end{equation}
\end{cor}
\begin{proof}
It follows directly from (\ref{equapremainder}).
\end{proof}
\begin{cor}\label{coruap-sqrt}
We have
\begin{equation}\label{equap-sqrt}
{u}_{ap}=\sqrt{A(1-x^2)}+\mathcal{O}\left((1-x^2)^{-2} \right)
\end{equation}
uniformly on
$[-1+\delta^{-1}(2A)^{-\frac{1}{5}},-1+2\delta]\cup[1-2\delta,1-\delta^{-1}(2A)^{-\frac{1}{5}}]$,
and
\begin{equation}\label{equap=sqrt}
{u}_{ap}=\sqrt{A(1-x^2)},\ \ x\in [-1+2\delta,1-2\delta].
\end{equation}
\end{cor}
\begin{proof}
In view of (\ref{equap}), if $x$ is not in the interpolation
intervals $[-1+\delta^{-1}(2A)^{-\frac{1}{5}},
-1+2\delta^{-1}(2A)^{-\frac{1}{5}}]$ and
$[1-2\delta^{-1}(2A)^{-\frac{1}{5}},
1-\delta^{-1}(2A)^{-\frac{1}{5}}]$, the assertions of the corollary
follow directly directly from the corresponding ones of Lemma
\ref{lemuout-sqrt}. For $x$ in the interpolation intervals, we also
have to use (\ref{equout-uin}).

The proof of the corollary is complete.
\end{proof}

\section{Linear analysis}\label{seclinear}
In this section, we will study the linearization of (\ref{eqEq})
about the approximate solution $u_{ap}$, namely the linear
Schr\"{o}dinger operator
\begin{equation}\label{eqLuap}
\mathcal{L}(\varphi)=-\varphi''+u_{ap}\varphi, \ \ \varphi\in
C^2(-1,1)\cap C[-1,1],\ \ \varphi(\pm 1)=0,
\end{equation}
(for convenience, we have divided by two).
\subsection{Properties of the potential of the Schr\"{o}dinger operator.}
In view of (\ref{equinner-}), we have
\[
(2A)^{-\frac{2}{5}}u_{in}\left(-1+(2A)^{-\frac{1}{5}}s
\right)=Y(s),\ \ 0\leq s \leq 2(2A)^\frac{1}{5}.
\]
Moreover, it follows from (\ref{equout-uin}) that
\[
(2A)^{-\frac{2}{5}}(\tilde{u}_{out}-u_{in})\left(-1+(2A)^{-\frac{1}{5}}s
\right)=\mathcal{O}(A^{-\frac{1}{5}})s^\frac{3}{2},
\]
uniformly on $\left[\delta^{-1},\delta(2A)^\frac{1}{5} \right]$, as
$A\to \infty$. Hence, via (\ref{equap}), we find that
\begin{equation}\label{equapCloc-}
(2A)^{-\frac{2}{5}}u_{ap}\left(-1+(2A)^{-\frac{1}{5}}s \right)\to
Y(s)\ \ \textrm{in}\ \ C_{loc}[0,\infty)\ \ \textrm{as}\ \ A\to
\infty.
\end{equation}
Analogously, we find that
\begin{equation}\label{equapCloc+}
(2A)^{-\frac{2}{5}}u_{ap}\left(1-(2A)^{-\frac{1}{5}}t \right)\to
Z(t)\ \ \textrm{in}\ \ C_{loc}[0,\infty)\ \ \textrm{as}\ \ A\to
\infty.
\end{equation}

The asymptotic behavior of $Y_{\pm}$ (recall (\ref{eqpainleve})) and
the definitions (\ref{equinner-})--(\ref{equinner+}) imply that
there exist constants $c,D>0$, independent of $A,\delta$, such that
\[
u_{in}\geq cA^\frac{1}{2}(1-x^2)^\frac{1}{2}
\]
for $x\in
\left[-1+D(2A)^{-\frac{1}{5}},-1+2\delta^{-1}(2A)^{-\frac{1}{5}}
\right]\cup
\left[1-2\delta^{-1}(2A)^{-\frac{1}{5}},1-D(2A)^{-\frac{1}{5}}
\right]$. Observe also that if $x\in
\left[-1+\delta^{-1}(2A)^{-\frac{1}{5}},1-\delta^{-1}(2A)^{-\frac{1}{5}}
\right]$ (which implies that $s,t\geq \delta^{-1}$), thanks to
(\ref{eqY-sqrt}) and (\ref{equoutmod}), we have
\[
   \tilde{u}_{out}^2  \geq  A(1-x^2)-CA^\frac{4}{5}\delta^\frac{3}{2} \geq  A(1-x^2)-CA\delta^\frac{5}{2}(1-x^2)\geq \frac{A}{2}(1-x^2),
\]
for some constant $C>0$ independent of both $A$ and $\delta$, where
we have decreased the value of $\delta$ if necessary. Combining the
above two relations with (\ref{equout-uin}) and (\ref{equap}), we
arrive at
\begin{equation}\label{equaplower}
u_{ap}\geq c \sqrt{\frac{A}{2}} \sqrt{1-x^2},\ \ x\in
\left[-1+D(2A)^{-\frac{1}{5}},1-D(2A)^{-\frac{1}{5}} \right].
\end{equation}
\subsection{Uniform a-priori estimates}
Let
\begin{equation}\label{eqnormweight}
\|\varphi \|_0\equiv
\|\varphi\|_{L^\infty(-1,1)}\equiv\sup_{(-1,1)}\left|\varphi(x)
\right|.
\end{equation}

The main result of this section is the following.

\begin{pro}\label{prouniform}
There exist constants $A_0,\ C>0$ such that, given $f\in C[-1,1]$,
there exists a unique classical solution to the boundary value
problem
\begin{equation}\label{eqBVP}
\mathcal{L}(\varphi)=f \ \ \textrm{in}\  \ (-1,1),\ \ \varphi(\pm
1)=0,
\end{equation}
and this solution satisfies
\begin{equation}\label{eqapriori0}
\|\varphi\|_0\leq CA^{-\frac{2}{5}}\|f\|_0,
\end{equation}
 provided that $A\geq A_0$.
\end{pro}
\begin{proof}
To establish existence and uniqueness for (\ref{eqBVP}), it suffices
to show the a-priori estimate (\ref{eqapriori0}) which implies that
the kernel of $\mathcal{L}$ is empty (see for example
\cite{walter}). Suppose that the latter estimate does not hold.
Then, there would exist sequences $A_n>0$, $\varphi_n \in
C^2[-1,1]$, $f_n\in C[-1,1]$ such that
\begin{equation}\label{eqconta1}
\mathcal{L}(\varphi_n)=f_n \ \ \textrm{in}\  \ (-1,1),\ \
\varphi_n(\pm 1)=0,
\end{equation}
\begin{equation}\label{eqconta2}
A_n\to \infty,\ \ \|\varphi_n\|_0=1, \ \ \textrm{and}\ \
A_n^{-\frac{2}{5}}\|f_n\|_0\to 0.
\end{equation}
Without loss of generality, we may assume that there are $x_n\in
(-1,1)$ such that
\[
\varphi_n(x_n)=\|\varphi_n\|_0=1,\ \ \varphi_n'(x_n)=0,\ \
\textrm{and}\ \ \varphi_n''(x_n)\leq 0,
\]
(otherwise we can consider $-\varphi_n$). Equation (\ref{eqconta1}),
for $x=x_n$, gives us that
\[
u_{ap}(x_n)\leq f_n(x_n).
\]
In view of (\ref{equaplower}) and (\ref{eqconta2}), we find that
\[\textrm{the}\  x_n\textrm{'s\ cannot\ be\ in} \
 \left(-1+2\delta^{-1}(2A_n)^{-\frac{1}{5}},1-2\delta^{-1}(2A_n)^{-\frac{1}{5}}\right)
\ \textrm{for large}\ n.
 \]
 Consequently, there are infinitely many
$n$'s such that
\[x_n\in\left(-1,-1+2\delta^{-1}(2A_n)^{-\frac{1}{5}}\right]\ \
\textrm{or}\ \
x_n\in\left[1-2\delta^{-1}(2A_n)^{-\frac{1}{5}},1\right).\] We may
assume, without loss of generality, that the former case occurs.
Therefore, abusing notation, we can choose a subsequence so that
\begin{equation}\label{eqxn} x_n\in
\left(-1,-1+2\delta^{-1}(2A_n)^{-\frac{1}{5}}\right],\ \ n\geq 1.
\end{equation}
Let
\[
\Phi_n(s)\equiv \varphi_n(x),\ \ F_n(s)\equiv f_n(x),\ \
x=-1+(2A_n)^{-\frac{1}{5}}s.
\]
Then, relations (\ref{eqconta1}) and (\ref{eqconta2}) become
\begin{equation}\label{eqconta3}
\Phi_n''-(2A_n)^{-\frac{2}{5}}u_{ap}\left(-1+(2A_n)^{-\frac{1}{5}}s
\right)\Phi_n=(2A_n)^{-\frac{2}{5}}F_n
\end{equation}
in $I_n\equiv\left[0, 2(2A_n)^\frac{1}{5}\right]$, $\Phi_n=0$ on the
boundary of $I_n$, and
\begin{equation}\label{eqconta4}
\| \Phi_n\|_{L^\infty(I_n)}=1,\ \
A_n^{-\frac{2}{5}}\|F_n\|_{L^\infty(I_n)}\to 0,
\end{equation}
respectively. Furthermore, recalling (\ref{eqxn}), we have that
\begin{equation}\label{eqsn}
\Phi_n(s_n)=1,\ \ \textrm{where}\ \ s_n\equiv
(2A_n)^\frac{1}{5}(x_n+1)\in (0,2\delta^{-1}].
\end{equation}
Making use of (\ref{equapCloc-}), (\ref{eqconta3}),
(\ref{eqconta4}), (\ref{eqsn}), and a standard diagonal compactness
argument, passing to a further subsequence, we find that
\[
\Phi_n\to \Phi_*\ \ \textrm{in}\ \ C^2_{loc}[0,\infty),\ \ s_n\to
s_*\in [0,2\delta^{-1}],
\]
where
\[
\Phi_*''-Y(s)\Phi_*=0\ \ \textrm{in}\ (0,\infty),\ \ \Phi_*(0)=0, \
\ \|\Phi_*\|_{L^\infty(0,\infty)}\leq 1,\ \textrm{and}\
\Phi_*(s_*)=1.
\]
On the other hand, by the non-degeneracy of $Y_\pm$ (recall
Proposition \ref{proYpm}), we arrive at a contradiction. We have
thus established the validity of (\ref{eqapriori0}).

The proof of the proposition is complete.
\end{proof}
\begin{rem}\label{remTunel}
Let $\mu_1^\pm<\mu_2^\pm<\cdots$,  with $\mu_i^\pm\to \infty$ as
$i\to \infty$, denote the eigenvalues of the linear operators
\[
\mathcal{M}_\pm(\psi)=-\psi''+Y_\pm (s)\psi
\]
with domain $\left\{\psi \in H^2(0,\infty), \ \sqrt{s}\psi\in
L^2(0,\infty),\ \psi(0)=0 \right\}$, which are self-adjoint in
$L^2(0,\infty)$ and have only simple eigenvalues in their spectrum
since $Y_\pm(s)\to \infty$ as $s\to \infty$ (see \cite{hislop} for
more details). It follows from Propositions \ref{proYpm} and
\ref{profelmer}  that \[\mu_1^-<0,\ \ \mu_2^->0,\ \ \textrm{while}
\ \
 \mu_1^+>0.\] In the case where $u_{ap}$ is even, using the obvious
notation, we denote the eigenvalues of the linear operators
$\mathcal{L}_\pm$ in (\ref{eqLuap}) by
$\lambda_1^\pm<\lambda_2^\pm<\cdots$. Arguing as in
\cite{pelinovsky,karalisourdisradial}, it follows readily that
\[
\lambda_{i+1}^\pm-\lambda_{i}^\pm=\mathcal{O}(A^{-k})  \ \
\textrm{and} \ \ \lambda_{i}^\pm = \mu_{i}^\pm
A^{-\frac{2}{5}}+\mathcal{O}\left(A^{-\frac{4}{5}} \right), \ \
i=1,3,5,\cdots, 2\left[\frac{n}{2}\right]+1,
\]
with $k,n\in \mathbb{N}$ fixed, as $A\to \infty$. The main
observation is that, because of the simplicity of the eigenvalues,
the associated (normalized) eigenfunction to $\lambda_{2m-1}$ is
even whereas that associated to $\lambda_{2m}$ is odd, for $m\geq
1$. Thus, the eigenvalue problem for $\mathcal{L}$ in $(-1,1)$
reduces to two eigenvalue problems in $(-1,0)$ with boundary
condition $\varphi(-1)=0$, $\varphi(0)=0$ and $\varphi(-1)=0$,
$\varphi'(0)=0$ respectively. The main point being that the
reduced eigenvalue problems have only one turning point (at
$x=-1$) and the proof of \cite[Prop. 3.25]{karalisourdisradial}
applies directly. We expect that, as in \cite{nakamura}, the
difference between two clustering eigenvalues is actually
exponentially small.
\end{rem}
\begin{rem}\label{remunbalanced}
In the case where $u_{ap}$ is nonsymmetric, one can adapt the
proof of \cite[Prop. 3.25]{karalisourdisradial} to show that the
corresponding linear operator $\mathcal{L}$ in (\ref{eqLuap}) has
only one negative eigenvalue, which satisfies $\lambda_{1} =
\mu_{1}^- A^{-\frac{2}{5}}+\mathcal{O}\left(A^{-\frac{4}{5}}
\right)$ as $A\to \infty$ (where $\mu_1^-$ as in Remark
\ref{remTunel}). However, it is not clear to us how to obtain
asymptotic expansions for the rest of the eigenvalues. Certainly
this has to depend on the  ordering between $\{\mu_i^-\}$ and
$\{\mu_i^+\}$.
\end{rem}

\begin{rem}\label{remlinearnearby}
It is easy to see that relation (\ref{eqapriori0}) as well as the
assertions of Remarks \ref{remTunel} and \ref{remunbalanced}
continue to hold if the potential of $\mathcal{L}$ was
$u_{ap}+\phi$ with $A^{-\frac{2}{5}}\|\phi\|_0\to 0$ as $A\to
\infty$ (clearly $\phi$ has to be even for the latter remark to
hold).
\end{rem}
\section{Existence of solutions and estimates}\label{secexist}
We seek solutions of (\ref{eqEq}) in the form
\begin{equation}\label{eqansatz}
u=u_{ap}+\phi.
\end{equation}
Substituting this ansatz in (\ref{eqEq}), and rearranging terms, we
see that $\phi$ solves
\begin{equation}\label{eqphi}
-2\phi''+2u_{ap}\phi=-\phi^2+2u_{ap}''-u_{ap}^2+A(1-x^2),\ \ x\in
(-1,1);\ \ \phi(\pm 1)=0.
\end{equation}

The next proposition is the main result of this section.

\begin{pro}\label{proexistPhi}
If $A$ is sufficiently large, there exists a constant $C>0$ and a
unique solution of (\ref{eqphi}) such that
\begin{equation}\label{eqphi0}
\|\phi\|_0\leq CA^\frac{1}{5}.
\end{equation}
\end{pro}
\begin{proof}
Let us write (\ref{eqphi}) in the abstract form
\begin{equation}\label{eqLNE}
2\mathcal{L}(\phi)=\mathcal{N}(\phi)+\mathcal{E};\ \ \phi(\pm 1)=0,
\end{equation}
where $\mathcal{L}$ was studied in Section \ref{seclinear}, \[
\mathcal{N}(\phi)\equiv-\phi^2,\] and $\mathcal{E}$ was defined in
(\ref{equapremainder}).

For $M>0$, consider the closed ball of $C[-1,1]$ that is defined by
\[
\mathcal{B}_M=\left\{\phi \in C[-1,1]\ :\ \|\phi \|_0\leq M
A^\frac{1}{5} \right\}.
\]
We will show that, if $M$ is chosen sufficiently large, the mapping
$\mathcal{T}:\mathcal{B}_M\to C^2[-1,1]$, defined by
\[
\mathcal{L}\left(\mathcal{T}(\phi)
\right)=\mathcal{N}(\phi)+\mathcal{E};\ \ \mathcal{T}(\phi)(\pm
1)=0,
\]
(recall Proposition \ref{prouniform}), maps $\mathcal{B}_M$ into
itself and is a contraction with respect to the $\|\cdot\|_0$ norm,
provided that $A$ is sufficiently large. Let $\phi \in
\mathcal{B}_M$, via (\ref{eqremainderuapGlobal}) and
(\ref{eqapriori0}), we have
\[\begin{array}{lll}
    \|\mathcal{T}(\phi)\|_0 & \leq & CA^{-\frac{2}{5}}\left(\|\mathcal{N}(\phi)\|_0+\|\mathcal{E}\|_0
\right) \\
      &   &   \\
      & \leq &  CA^{-\frac{2}{5}}M^2
A^\frac{2}{5}+CA^{-\frac{2}{5}}A^\frac{3}{5}
 \\
      &   &   \\
     & \leq & CA^\frac{1}{5}(M^2A^{-\frac{1}{5}}+1),
  \end{array}
\]
where $C>0$ is independent of both large $A$ and $M$. By virtue of
the above relation, we can choose a large $M>0$ such that
$\mathcal{T}$ maps $\mathcal{B}_M$ into itself, for all sufficiently
large $A$. From now on, we fix such an $M$. Similarly, for $\phi_1,
\phi_2\in \mathcal{B}_M$, we have
\[
\|\mathcal{T}(\phi_1)-\mathcal{T}(\phi_2)\|_0\leq
CA^{-\frac{2}{5}}A^\frac{1}{5}\|\phi_1-\phi_0\|_0=CA^{-\frac{1}{5}}\|\phi_1-\phi_0\|_0,
\]
which implies that, for large $A$, the mapping
$\mathcal{T}:\mathcal{B}_M \to \mathcal{B}_M$ is a contraction.
Hence, by Banach's fixed point theorem, we infer that $\mathcal{T}$
has a unique fixed point in the closed set $\mathcal{B}_M$. In turn,
this furnishes a solution of (\ref{eqphi}) which satisfies the
uniform estimate (\ref{eqphi0}).

The proof of the proposition is complete.
\end{proof}

\begin{rem}\label{remeven}
If $u_{ap}$ is even, we can of course restrict ourselves to even
fluctuations $\phi$ in (\ref{eqansatz}).
\end{rem}

In the next two lemmas we will show that estimate (\ref{eqphi0})
can be improved away from the boundary points.

\begin{lem}\label{lemphi'}
Let $\phi$ be as in Proposition \ref{proexistPhi}. Given $L\geq 1$,
there exists a constant $C_L>0$ such that
\begin{equation}\label{eqphi'}
|\phi'(x)|\leq C_LA^\frac{2}{5},\ \ x\in
\left[-1,-1+(2A)^{-\frac{1}{5}}L\right]\bigcup\left[1-(2A)^{-\frac{1}{5}}L,1\right],
\end{equation}
for all $A$ sufficiently large.
\end{lem}
\begin{proof}
Let
\begin{equation}\label{eqPsinDef}
\Psi(s)=(2A)^{-\frac{1}{5}}\phi\left(-1+(2A)^{-\frac{1}{5}}s
\right),\ \ s\in \left[0,2(2A)^\frac{1}{5} \right].
\end{equation}
From (\ref{eqphi}), we find that
\begin{equation}\label{eqPsinEq}
-\Psi''+(2A)^{-\frac{2}{5}}u_{ap}\left(-1+(2A)^{-\frac{1}{5}}s
\right)\Psi= -\frac{1}{2}(2A)^{-\frac{1}{5}}
\Psi^2+\frac{1}{2}(2A)^{-\frac{3}{5}}\mathcal{E}\left(-1+(2A)^{-\frac{1}{5}}s
\right),
\end{equation}
for $s\in \left(0,2(2A)^\frac{1}{5} \right)$, and $ \Psi(0)=0$.
Furthermore, from (\ref{eqphi0}), and (\ref{eqPsinDef}), it follows
that
\begin{equation}\label{eqPsiLinfty}
\|\Psi\|_{L^\infty\left(0,2(2A)^\frac{1}{5} \right)}\leq C.
\end{equation}
In turn,  relations (\ref{eqremainderuapGlobal}),
(\ref{equapCloc-}), (\ref{eqPsinEq}) and (\ref{eqPsiLinfty}) imply
that, given $L\geq 1$, there exists a constant $C_L>0$ such that
\[
\left|\Psi''(s) \right|\leq C_L \ \ \textrm{on}\ [0,L],
\]
provided that $A$ is sufficiently large. Consequently, it follows
from (\ref{eqPsiLinfty}), the above relation, and the elementary
interpolation inequality
\[
\|\Psi'\|_{L^\infty(0,L)}\leq 2
\|\Psi\|_{L^\infty(0,L)}+\|\Psi''\|_{L^\infty(0,L)}
\]
(keep in mind that $L\geq 1$), that
\[
\left|\Psi'(s) \right|\leq C_L \ \ \textrm{on}\ [0,L],
\]
provided that $A$ is sufficiently large (for some possibly larger
constant $C_L$). Now, the validity of estimate (\ref{eqphi'}) on the
interval $\left[-1,-1+(2A)^{-\frac{1}{5}}L\right]$ follows directly
via (\ref{eqPsinDef}). Analogously we can show its validity on the
interval $\left[1-(2A)^{-\frac{1}{5}}L,1\right]$.

The proof of the lemma is complete.
\end{proof}

\begin{lem}\label{lemdecay}
Let $\phi$ be as in Proposition \ref{proexistPhi}. There exist
positive constants $C,D$ such that
\begin{equation}\label{eqphidecay}
\left|\phi(x) \right|\leq C(1-x^2)^{-1},\ \ x\in
\left[-1+(2A)^{-\frac{1}{5}}D,1-(2A)^{-\frac{1}{5}}D \right],
\end{equation}
provided that $A$ is sufficiently large.
\end{lem}
\begin{proof}
Let
\[
\underline{\psi}(x)=-K(1-x^2)^{-1},\ \ x\in
\left[-1+(2A)^{-\frac{1}{5}}D,1-(2A)^{-\frac{1}{5}}D \right],
\]
with constant $K>0$ to be determined, and $D>0$ to be chosen larger
than that  in (\ref{equaplower}), such that $\underline{\psi}$ is a
lower solution to (\ref{eqphi}) on the above interval.
Differentiating twice gives us
\[
\underline{\psi}''=-K(6x^2+2)(1-x^2)^{-3}\geq -8K(1-x^2)^{-3}.
\]
Recalling (\ref{equapremainder}), (\ref{equaplower}), and
(\ref{eqphi0}), we find that
\[\begin{array}{lll}
    -2\underline{\psi}''+2u_{ap}\underline{\psi}+\underline{\psi}^2-\mathcal{E} & \leq & 16K(1-x^2)^{-3}-cA^\frac{1}{2}K(1-x^2)^{-\frac{1}{2}}+CA^\frac{2}{5}\\ & &+CA^\frac{1}{2}(1-x^2)^{-\frac{1}{2}}
 \\
      &   &   \\
      & \leq & (1-x^2)^{-\frac{1}{2}}\left[CK(1-x^2)^{-\frac{5}{2}}-cKA^\frac{1}{2}+CA^\frac{1}{2}
\right] \\
      &   &   \\
      & \leq & (1-x^2)^{-\frac{1}{2}}\left[CKD^{-\frac{5}{2}}A^\frac{1}{2}-cKA^\frac{1}{2}+CA^\frac{1}{2}
\right]
 \\
      &   &   \\
      & \leq & (1-x^2)^{-\frac{1}{2}}\left[-\frac{1}{2}cKA^\frac{1}{2}+CA^\frac{1}{2}
\right], \\
      \end{array}
\]
where the constants $c,C$ are independent of both $A$ and $D$,
having increased the value of $D$ if necessary. Hence, we can chose
a large $K>0$ such that
\[
-2\underline{\psi}''+2u_{ap}\underline{\psi}+\underline{\psi}^2-\mathcal{E}\leq
0,\ \ x\in \left[-1+(2A)^{-\frac{1}{5}}D,1-(2A)^{-\frac{1}{5}}D
\right],
\]
provided that $A$ is sufficiently large.  By virtue of
(\ref{eqphi}), (\ref{eqphi0}), the above equation, and making use of
the maximum principle, we deduce that
\[
-C(1-x^2)^{-1}\leq \phi(x),\ \ x\in
\left[-1+(2A)^{-\frac{1}{5}}D,1-(2A)^{-\frac{1}{5}}D \right],
\]
for some large constant $C>0$ and all large $A$. Analogously we can
establish the other side of the desired estimate (\ref{eqphidecay}).

The proof of the lemma is complete.
\end{proof}

In summary, we have the following.
\begin{pro}\label{prosummuap}
There exists a solution of (\ref{eqEq}) such that
\[\begin{array}{ll}
    u-u_{ap}=\mathcal{O}(A^\frac{2}{5})(1-x^2), & x\in \left[-1,-1+(2A)^{-\frac{1}{5}}D\right]\bigcup
\left[1-(2A)^{-\frac{1}{5}}D,1\right], \\
     &  \\
    u-u_{ap}=\mathcal{O}(1)(1-x^2)^{-1}, & x\in \left[-1+(2A)^{-\frac{1}{5}}D,1-(2A)^{-\frac{1}{5}}D\right],
  \end{array}
\]
for some constant $D\gg 1$, uniformly as $A\to \infty$.
\end{pro}
\section{Proof of the main result}\label{secmainResult}
From Propositions \ref{prouout} and \ref{prosummuap}, relation
(\ref{equap}), Corollary \ref{coruap-sqrt} and Remark
\ref{remlinearnearby}, we can infer the validity of  Theorem
\ref{thmex}.

\appendix
\section{Uniqueness and non-degeneracy of solutions for problem (\ref{eqeta})}\label{appenFelmer}
\begin{pro}\label{profelmer}
Problem (\ref{eqeta}) has $Y_+-Y_-$ as its unique solution.
Moreover, this solution is non-degenerate, namely there are no
nontrivial bounded solutions to $(\ref{eqnondegYpm})_-$.
\end{pro}
\begin{proof}
To show uniqueness, we argue by contradiction and assume that there
exist two distinct solutions $u_1$ and $u_2$ of (\ref{eqeta}). As in
\cite{byeonOshita,felmerUniqueness,tanaka}, the solutions $u_1$ and
$u_2$ can be chosen such that $u_1'(0)<u_2'(0)$ and such that they
intersect \emph{at most once} in $(0,\infty)$ (this is achieved by a
shooting argument, making use only of the smooth dependence on the
initial data of solutions to the ordinary differential equation in
(\ref{eqeta})). Under this assumption, as in
\cite{byeonOshita,felmerUniqueness,tanaka}, we have that
\begin{equation}\label{equ1overu2}
\frac{d}{ds}\left(\frac{u_1(s)}{u_2(s)} \right)>0,\ \ s>0.
\end{equation}
We point out that, in the above calculation, the terms involving
$Y_+$ cancel each other, and thus the form of $Y_+$ is irrelevant
for this part of the proof. Furthermore, if we define
\[
E(s;u)=\left[u'(s)\right]^2-Y_+(s)u^2(s)+\frac{1}{3}u^3(s),\ \ s>0,\
u\in C^2\left([0,\infty) \right),
\]
a direct calculation yields that
\begin{equation}\label{eqE'}
\frac{d}{ds}E(s;u_i)=-Y_+'u_i^2<0,\ \ s>0,\ \ i=1,2.
\end{equation}
Therefore, using the standard fact that any solution of
(\ref{eqeta}) decays super-exponentially as $s\to \infty$, we obtain
that
\begin{equation}\label{eqEpositive}
E(s;u_i)>\lim_{s\to \infty}E(s;u_i)=0,\ \ s>0,\ \ i=1,2.
\end{equation}
Next, as in \cite{yotsutani}, we set
\[
F(s)=E(s;u_2)-\left(\frac{u_2}{u_1} \right)^2E(s;u_1),\ \ s>0.
\]
Note that, thanks to l'hospital's rule, we have $F(0)=0$. A direct
calculation, making use of (\ref{eqE'}), yields that
\[
F'(s)=-\frac{d}{ds}\left\{\left(\frac{u_2}{u_1} \right)^2
\right\}E(s;u_1),\ \ s>0.
\]
 So, in view of (\ref{equ1overu2}) and
(\ref{eqEpositive}), we get that
\[
F'(s)>0,\ \ s>0.
\]
Consequently, noting that (\ref{equ1overu2}) implies that
\[
0<\frac{u_2(s)}{u_1(s)}<\frac{u_2'(0)}{u_1'(0)},\ \ s>0,
\]
and making once more use of the super-exponential decay of $u_1$ and
$u_2$, we arrive at the \emph{strict} inequality
\[
0=F(0)<\lim_{s\to \infty}F(s)=0,
\]
which is a contradiction. Hence, problem (\ref{eqeta}) has $Y_+
-Y_-$ as its only solution.

With some care, the non-degeneracy of $\varphi$ can also be shown as
in \cite{felmerUniqueness} (see also \cite{byeonOshita},
\cite{tanaka}). The fact that $V(s)=Y_+(s)\to \infty$, as $s\to
\infty$, poses an obstruction in adapting some proofs of
\cite{felmerUniqueness} to our setting (especially the second part
of the proof of Proposition 3.1 therein). Nevertheless, the fact
that $Y_+'$ is positive on $[0,\infty)$ and decays to zero at an
algebraic rate, see (\ref{eqY'}), will allow us to bypass some of
the arguments in \cite{felmerUniqueness}, and in fact provide a more
direct proof as follows. Firstly, motivated from \cite{byeonOshita},
we define
\[
\|\phi\|=\left(\int_{0}^{\infty}\left[(\phi')^2+Y_+(s)\phi^2
\right]ds \right)^\frac{1}{2},
\]
and let $X$ be the completion of $C_0^\infty(0,\infty)$ with respect
to $\|\cdot\|$. We note that
\[
\|\phi\|^2\geq \mu_1\int_{0}^{\infty}\phi^2ds,
\]
where $\mu_1>0$ is the principal eigenvalue of
\[
-\psi''+Y_+(s)\psi=\mu\psi,\ \ s>0,\ \ \psi(0)=0,\ \psi\in
L^2(\mathbb{R}).
\]
Let $\varphi=Y_+ -Y_-$ be the unique solution of (\ref{eqeta}), then
$\varphi$ is a critical point of the functional
\[
I(u)=\int_{0}^{\infty}\left(|u'|^2+Y_+(s)u^2-\frac{1}{3}u_+^3
\right)ds,
\]
where $I:X\to \mathbb{R}$ is of class $C^2$ (here
$u_+=\max\{u,0\}$). This functional has the mountain pass structure
(see for instance \cite{malchiodicambridge}), and the unique
solution $\varphi$ of (\ref{eqeta}) corresponds to a mountain pass
solution. We point out that, even though the interval $(0,\infty)$
is infinite, compactness is restored by the property that $Y_+(s)\to
\infty$ as $s\to \infty$ (see \cite{rabinowitz}). We define the
Morse index of $\varphi$ as
\[\begin{split}
i(I,\varphi)=\max\left\{\textrm{dim}H \ :\ H\subset X\ \textrm{is\
a\ subspace\ such\ that}\right. \\ \left. I''(\varphi)(h,h)<0\
\textrm{for\ all}\ h\in H\backslash \{0\} \right\}. \end{split}\]
It follows from the general theorem in \cite{hofer} that
\begin{equation}\label{eqhofer}i(I,\varphi)\leq 1.\end{equation} In fact, for the specific equation, this can be shown in an elementary way (see \cite{berestWei}). As in
\cite{byeonOshita,felmerUniqueness,tanaka}, we introduce a
perturbed functional
\[
J_\delta(u)=I(u)-\delta\int_{0}^{\infty}\left(\frac{1}{3}u_+^3-\frac{1}{2}\varphi(s)u^2
\right)ds,\ \ u\in X,
\]
for small $\delta>0$. By the maximum principle, we see that
non-trivial critical points of $J_\delta$ are solutions to the
problem
\begin{equation}\label{eqNLSdelta}\left\{\begin{array}{l}
                                    2u''-\left(2Y_+(s)+\delta\varphi(s) \right)u+(1+\delta)u^2=0,\ \
s>0,\ \ u(s)>0,\ s>0,
 \\
                                      \\
                                    u(0)=0,\ \ \lim_{s\to \infty}u(s)=0.
                                  \end{array}\right.
\end{equation}
Observe that $\varphi$ is a solution of (\ref{eqNLSdelta}) for all
$\delta>0$. As in \cite{felmerUniqueness}, our primary objective is
to apply the arguments that were used for showing uniqueness for
(\ref{eqeta}) in order to infer that $\varphi$ is the only solution
of (\ref{eqNLSdelta}) if $\delta>0$ is sufficiently small. These
arguments can be applied almost word for word to (\ref{eqNLSdelta}),
once we show that the corresponding relation to (\ref{eqE'}) holds.
In other words, we have to show that
\begin{equation}\label{eqfelmerposi}
2Y_+'(s)+\delta \varphi'(s)>0,\ \ s>0,
\end{equation}
for sufficiently small $\delta>0$ (under the assumptions of
\cite{felmerUniqueness}, recall our discussion following
(\ref{eqfelmerGener}), this was not possible and the authors had to
argue indirectly). To this end, note that we have the following
rough estimates:
\[
Y_+'(s)\geq \min\left\{c,\frac{1}{4}s^{-\frac{1}{2}} \right\},\ \
\left|\varphi'(s) \right|\leq Ce^{-s},\ \ s\geq0,
\]
for some positive constants $c,C$ (the former estimate holds via
(\ref{eqY'}), while the latter from the super-exponential decay of
$\varphi$ and (\ref{eqeta})). We therefore deduce that
(\ref{eqfelmerposi}) is valid if $\delta\in (0,m)$, where $m>0$ is
the minimum value of the function
\[
2C^{-1}\min\left\{c,\frac{1}{4}s^{-\frac{1}{2}} \right\}e^s,\ \
s\geq 0.
\]
Consequently, if $\delta>0$ is sufficiently small, the function
$\varphi=Y_+-Y_-$ is the only solution to (\ref{eqNLSdelta}).

As in \cite{byeonOshita,felmerUniqueness,tanaka}, in order to show
that the unique solution $\varphi$ of (\ref{eqeta}) is
non-degenerate, we will argue by contradiction. So, assume that
$\varphi$ is degenerate. In view of (\ref{eqhofer}), this implies
that there exists a $2$-dimensional subspace $H\subset X$ such that
\[
I''(\varphi)(h,h)\leq 0\ \ \textrm{for\ all}\ h\in H.
\]
Since
\[
J_\delta''(u)(h,h)=I''(u)(h,h)-\delta\int_{0}^{\infty}\left(2u_+-\varphi(s)
\right)h^2ds,
\]
for any $u\in X,\ h\in H$, we have
\[
J''_\delta(\varphi)(h,h)=I''(\varphi)(h,h)-\delta\int_{0}^{\infty}\varphi
h^2ds.
\]
In particular, we see that
\[
J''_\delta(\varphi)(h,h)<0\ \ \textrm{for\ all}\ h\in H\backslash
\{0\},
\]
which implies that $i(J_\delta,\varphi)\geq 2$. On the other hand,
since $J_\delta$ has the mountain pass structure, and
(\ref{eqNLSdelta}) has $\varphi$ as its only solution for small
$\delta>0$, we must have $i(J_\delta,\varphi)\leq 1$ for small
$\delta>0$. We have therefore arrived at a contradiction, thus
completing the proof of the non-degeneracy of $\varphi$.

The proof of the proposition is complete.
\end{proof}

  \textbf{Acknowledgment.} The research leading
to these results has received funding from the European Union's
Seventh Framework Programme $(\textrm{FP7-REGPOT}-2009-1)$ under
grant agreement $\textrm{n}^\textrm{o}$ 245749 and by the ARISTEIA
(Excellence) programme ``Analysis of discrete, kinetic and
continuum models for elastic and viscoelastic response'' of the
Greek Secretariat of Research.


\begin{thebibliography}{50}
\bibitem{malchiodicambridge}
{\sc A. Ambrosetti}, and {\sc A. Malchiodi}, Nonlinear analysis and
semilinear elliptic problems, Cambridge studies in advanced
mathematics \textbf{104}, Cambridge university press, 2007.


\bibitem{berestWei}
{\sc H. Berestycki}, and {\sc J. Wei},  \emph{On least energy
solutions to a semilinear elliptic equation in a strip}, Discrete
Contin. Dyn. Syst. \textbf{28} (2010), 1083-1099.


\bibitem{byeonOshita}
{\sc J. Byeon}, and {\sc Y. Oshita}, \emph{Uniqueness of standing
waves for nonlinear Schr\"{o}dinger equations}, Proceedings of the
Royal Society of Edinburgh \textbf{138A} (2008), 975-987.


\bibitem{dancer-lazer}
{\sc E. N. Dancer}, and {\sc S. Yan}, \emph{On the superlinear
Lazer--McKenna conjecture}, J. Differential Equations \textbf{210}
(2005), 317-351.

\bibitem{danceryanCrtitic}
{\sc E. N. Dancer}, and {\sc S. Yan}, \emph{On the Lazer-McKenna
conjecture involving critical and supercritical exponents},
Methods Appl. Anal.  \textbf{15} (2008), 97-119.

\bibitem{delPinoIndi} {\sc M. del Pino}, and {\sc P. L. Felmer},
\emph{Spike-layered solutions of singularly perturbed elliptic
problems in a degenerate setting}, Indiana Univ. Math. J.
\textbf{48} (1999),  883-898.

\bibitem{delpinocpam}
{\sc M.  del Pino},  {\sc M. Kowalczyk}, and {\sc J. Wei},
\emph{Concentration on curves for nonlinear Schr\"{o}dinger
equations}, Comm. Pure Appl. Math. \textbf{60} (2007),  113--146.


\bibitem{felmerUniqueness}
{\sc P. Felmer}, {\sc S. Mart\'{i}nez}, and {\sc K. Tanaka},
\emph{Uniqueness of radially symmetric positive solutions for
$-\Delta u +u = u^p$ in an annulus}, J. Differential Equations
\textbf{245} (2008), 1198-1209.

\bibitem{fokas}
{\sc A. S. Fokas}, {\sc A. R. Its}, {\sc A. A. Kapaev}, and {\sc
V. Y. Novokshenov}, Painlev\'{e} Transcendents: The
Riemann-Hilbert Approach,
Amer. Math. Soc., Providence, RI, 2006.


\bibitem{pelinovsky}
{\sc C. Gallo}, and {\sc D. Pelinovsky}, \emph{On the
Thomas--Fermi ground state in a harmonic potential}, {Asymptot.
Anal.} \textbf{73} (2011), 53--96.


\bibitem{hastingsTroy}
{\sc S. P. Hastings}, and {\sc W. C. Troy}, \emph{On some
conjectures of Turcotte, Spence, Bau, and Holmes}, SIAM J. Math.
Anal. \textbf{20} (1989), 634--642.


\bibitem{hastingsBOOK}
{\sc S. P. Hastings}, and {\sc J. B. McLeod}, Classical methods in
ordinary differential equations with applications to boundary
value problems, Graduate studies in mathematics \textbf{129},
American Mathematical Society, (2010).

\bibitem{hislop}
{\sc P. D.  Hislop}, and {\sc I. M. Sigal}, Introduction to
spectral theory with applications to Schr\"{o}dinger operators,
Applied mathematical sciences \textbf{113}, Springer-Verlag, New
York, 1996.


\bibitem{hofer}
{\sc H. Hofer}, \emph{A note on the topological degree at a critical
point of mountain pass type}, Proc. Amer. Math. Soc. \textbf{90}
(1984), 309--315.

\bibitem{holmes}
{\sc P. Holmes},  \emph{On a second-order boundary value problem
arising in combustion theory}, Quart. Appl. Math. \textbf{40}
(1982/83), 53-62.



\bibitem{holmespainleve}
{\sc P. Holmes}, and {\sc D. Spence}, \emph{On a Painlev\'{e}-type
boundary-value problem},  Quart. J. Mech. Appl. Math. \textbf{37}
(1984), 525--538.

\bibitem{tanaka}
{\sc Y. Kabeya}, and {\sc K. Tanaka}, \emph{Uniqueness of positive
radial solutions of semilinear elliptic equations in $\mathbb{R}^N$
and S\'{e}r\'{e}s non-degeneracy condition}, Comm. Partial
Differential Equations \textbf{24} (1999), 563-598.

\bibitem{karalisourdisradial}
 {\sc G. Karali}, and  {\sc C. Sourdis}, \emph{Radial and bifurcating
non-radial solutions for a singular perturbation problem in the
case of exchange of stabilities},  {Ann. Inst. H. Poincar\'{e}
Anal. Non Lin\'{e}aire} \textbf{29} (2012), 131-170.

\bibitem{karalisourdisresonance}
{\sc G. Karali}, and  {\sc C. Sourdis}, \emph{Resonance phenomena
in a singular perturbation problem in the case of exchange of
stabilities}, Comm. Partial Differential Equations \textbf{37}
(2012), 1620--1667.


\bibitem{karaliGS}
{\sc G. Karali}, and  {\sc C. Sourdis}, \emph{The ground state of
a Gross-Pitaevskii energy with general potential in the
Thomas-Fermi limit}, Arch. Ration. Mech. Anal. (2015), (DOI)
10.1007/s00205-015-0844-3.

\bibitem{yotsutani}
{\sc N. Kawano}, {\sc E. Yanagida}, and {\sc S. Yotsutani},
\emph{Structure theorems for positive radial solutions to $\Delta u
+K(|x|)u^p=0$ in $\mathbb{R}^N$}, Funkcial. Ekvac. \textbf{36}
(1993), 557--579.


\bibitem{malchiodiCPAM} {\sc A. Malchiodi}, and {\sc M. Montenegro}, \emph{Boundary concentration phenomena for a singularly perturbed elliptic
problem}, Comm. Pure Appl. Math. \textbf{55} (2002), 1507-1568.



\bibitem{miller} {\sc P. D. Miller}, Applied asymptotic analysis,  Graduate studies
in mathematics \textbf{75}, American Mathematical Society, (2006).


\bibitem{nakamura}
{\sc S. Nakamura}, \emph{A Remark on eigenvalue splittings for
one-dimensional double-well Hamiltonians}, Letters in Math. Phys.
\textbf{11} (1986), 337--340.

\bibitem{rabinowitz}
{\sc P. H. Rabinowitz}, \emph{On a class of nonlinear
Schr\"{o}dinger equations}, Z. Angew. Math. Phys. \textbf{43}
(1992), 270-291.


\bibitem{sattinger}
{\sc D. H. Sattinger}, Topics in stability and bifurcation theory,
{Lecture Notes in Math.} \textbf{309}, Springer, Heidelberg, 1973.

\bibitem{schecter-sourdis}
{\sc S. Schecter}, and {\sc C. Sourdis}, \emph{Heteroclinic orbits
in slow-fast Hamiltonian systems with slow manifold bifurcations},
J. Dyn. Diff. Equat.  \textbf{22} (2010),  629-655.

\bibitem{sourdis-fife}
{\sc C. Sourdis}, and {\sc P. C. Fife}, \emph{Existence of
heteroclinic orbits for a corner layer problem in anisotropic
interfaces}, {Adv. Differential Equations} \textbf{12} (2007),
623-668.


\bibitem{jo}
{\sc S. K. Tin}, {\sc N. Kopell}, and {\sc C. K. R. T. Jones},
\emph{Invariant manifolds and singularly perturbed boundary value
problems}, SIAM J. Num. Anal. \textbf{31} (1994), 1558--1576.

\bibitem{turcotte}
{\sc D. L. Turcotte},  {\sc D. A. Spence},  and {\sc H. H. Bau},
\emph{Multiple solutions for natural convective flows in an
internally heated, vertical channel with viscous dissipation and
pressure work}, Int. J. Heat Mass Transfer \textbf{25} (1982),
699--706.

\bibitem{walter}
{\sc W. Walter}, Ordinary differential equations, Graduate texts in
mathematics \textbf{182}, Springer-Verlag, New York, 1998.
\end{thebibliography}
\end{document}